\documentclass[final,3p,times]{elsarticle}

\usepackage{amssymb}

\usepackage{amsmath}

\journal{Statistics \& Probability Letters}

\begin{document}

\begin{frontmatter}



\title{An explicit representation of Verblunsky coefficients}


\author[label1]{N.~H.~Bingham}
\author[label2]{Akihiko Inoue\corref{cor1}}
\ead{inoue100@hiroshima-u.ac.jp}
\author[label3]{Yukio Kasahara}

\cortext[cor1]{Corresponding author}

\address[label1]{Department of Mathematics, Imperial College London,
London SW7 2AZ, UK}

\address[label2]{Department of Mathematics, 
Hiroshima University, 
Higashi-Hiroshima 739-8526, Japan}

\address[label3]{Department of Mathematics,
Hokkaido University,
Sapporo 060-0810, Japan}

\begin{abstract}
We prove a representation of 
the partial autocorrelation function (PACF) of a stationary process, 
or of the Verblunsky coefficients of its normalized spectral measure, 
in terms of the Fourier coefficients of the phase function. 
It is not of fractional form, whence simpler than the existing one obtained by the second author. 
We apply it to 
show a general estimate on the Verblunsky coefficients for short-memory processes as well as 
the precise asymptotic behaviour, with remainder term, of those 
for FARIMA processes.
\end{abstract}


\end{frontmatter}


\newtheorem{thm}{Theorem}[section]
\newtheorem{prop}[thm]{Proposition}
\newtheorem{lem}[thm]{Lemma}

\newdefinition{remark}{Remark}
\newdefinition{defn}{Definition}
\newdefinition{ex}{Example}

\newproof{proof}{Proof}

\numberwithin{equation}{section}
\renewcommand{\theenumi}{\roman{enumi}}
\renewcommand{\labelenumi}{(\theenumi)}




\section{Introduction}\label{sec:1}

Let $\{X_n: n\in\mathbb{Z}\}$ be a real, zero-mean, weakly stationary 
process, defined on a probability space $(\Omega,\mathcal{F},P)$, 
with spectral measure not of finite support, which we 
shall simply call a \textit{stationary process}\/. 
Here the spectral measure is the finite measure $\mu$ on $(-\pi,\pi]$ 
in the spectral representation 
$\gamma(n)=\int_{(\pi,\pi]}e^{in\theta}\mu(d\theta)$ of the 
\textit{autocovariance function}\/ $\gamma(n):=E[X_nX_0]$, $n\in\mathbb{Z}$. 
For $\{X_n\}$, we have another sequence 
$\{\alpha(n)\}_{n=1}^{\infty}$ 
called the \textit{partial autocorrelation function}\/ (PACF); 
see (\ref{eq:2.1}) below for the definition. 
In the theory of orthogonal polynomials on the unit circle (OPUC), however, 
the PACF $\{\alpha(n)\}_{n=1}^{\infty}$ appears as the sequence of 
\textit{Verblunsky coefficients}\/ of the normalized spectral measure 
$\tilde{\mu}:=(\mu(-\pi,\pi])^{-1}\mu$. 
Notice that $(-\pi,\pi]$ can be identified with the unit circle 
$\mathbb{T}:=\{z\in\mathbb{C}\ \vert\ 
\vert z\vert=1\}$ by the map $\theta\mapsto e^{i\theta}$, whence 
$\mu$ or $\tilde{\mu}$ with a measure on $\mathbb{T}$. 
For a survey of OPUC, see Simon (2005a, 2005b, 2005c, 2011).

The Verblunsky coefficients $\{\alpha(n)\}_{n=1}^{\infty}$ 
give an unrestricted parametrization of the normalized spectral measure 
$\tilde{\mu}$ of $\{X_n\}$, 
in that the only inequalities 
restricting the $\alpha(n)$ are $\alpha(n) \in [-1,1]$, 
or $\alpha(n) \in (-1,1)$ in the non-degenerate case relevant here. 
This result is due to Barndorff--Nielsen and Schou (1973), 
Ramsey (1974) in the time-series context. 
However, in OPUC, the result 
dates back to Verblunsky (1935, 1936). 
See, e.g., Simon (2005b, 2005c) and Bingham (2011) for background.

The aim of this paper is to prove an explicit representation of 
the Verblunsky coefficients $\{\alpha(n)\}_{n=1}^{\infty}$ in terms of 
another sequence $\{\beta_n\}_{n=0}^{\infty}$ 
defined by
\begin{equation}
\beta_n:=\sum\nolimits_{v=0}^{\infty} c_v a_{v+n}, \qquad n=0,1,\dots, 
\label{eq:1.1}
\end{equation}
where $\{c_n\}_{n=0}^{\infty}$ and 
$\{a_n\}_{n=0}^{\infty}$ are the MA and AR coefficients 
of $\{X_n\}$, respectively, defined by (\ref{eq:2.3}) below. 
See Inoue and Kasahara (2004 Section 3, 2006 (2.23)). 
We notice that $\beta_n$ correspond to the 
Fourier coefficients of the \textit{phase function\/}
of the process (see Remark \ref{rem:1} in \S 2). 
The proof of the representation of $\{\alpha(n)\}$ is based on the result of 
Inoue and Kasahara (2006) on the explicit 
representation of finite predictor coefficients as well as the Levinson 
(or Levinson--Durbin) algorithm or the Szeg\"o recursion. 
The algorithm is due to Szeg\"o (1939), Levinson (1947), and 
Durbin (1960); 
for a textbook 
account, see Pourahmadi (2001, Section 7.2).

We notice that, in Inoue (2008), the second author already proved a 
representation of $\{\alpha(n)\}_{n=1}^{\infty}$ in terms of $\{\beta_n\}_{n=0}^{\infty}$. 
However, the representation of $\{\alpha(n)\}_{n=1}^{\infty}$ 
in the present paper is much simpler than that in Inoue (2008), in that 
the latter is of fractional form while the former not. 
We apply the result to show a general estimate of $\alpha(n)$ for short-memory processes 
as well as the precise asymptotic behaviour, with remainder, of $\alpha(n)$ 
for FARIMA processes. 
The FARIMA model is a popular parametric model with long memory, and 
was introduced independently by Granger and Joyeux (1980) and Hosking (1981). 
See Brockwell and Davis (1991, Section 9) for textbook treatment. 
The long memory of the FARIMA model comes from the singularity at zero of its spectral density.

In \S 2, we state the main result, i.e., the 
representation of the Verblunsky coefficients. Its proof is given in \S 3. 
In \S 4, we apply the main result to both short-memory and FARIMA processes.

\section{Main result}\label{sec:2}

Let $H$ be the real Hilbert space spanned by $\{X_k:k\in\mathbb{Z}\}$ 
in $L^2(\Omega,\mathcal{F},P)$, which has inner product 
$(Y_1,Y_2):=E[Y_1Y_2]$ and norm $\Vert Y\Vert:=(Y,Y)^{1/2}$. 
For an interval $I\subset \mathbb{Z}$, we write $H_I$ for 
the closed subspace of $H$ spanned by $\{X_k: k\in I\}$ 
and $H_{I}^{\bot}$ for the orthogonal complement of $H_I$ in $H$. 
Let $P_I$ and $P^{\bot}_I$ be the orthogonal projection operators of $H$ 
onto $H_I$ and $H_{I}^{\bot}$, respectively. 
The projection $P_IY$ stands for 
the best linear predictor of $Y$ based on the observations 
$\{X_k: k\in I\}$, and $P_I^{\bot}Y$ for its prediction error. 

The PACF $\{\alpha(n)\}_{n=1}^{\infty}$ of 
$\{X_n\}$ is defined by
\begin{equation}
\alpha(1):=\frac{\gamma(1)}{\gamma(0)},\qquad 
\alpha(n)
:=\frac{(P^{\bot}_{[1,n-1]}X_n, P^{\bot}_{[1,n-1]}X_0)}{\Vert P^{\bot}_{[1,n-1]}X_n\Vert^2},
\qquad n=2,3,\dots
\label{eq:2.1}
\end{equation}
(cf.\ Brockwell and Davis (1991, Sections 3.4 and 5.2)).
As stated in \S 1, The PACF $\{\alpha(n)\}_{n=1}^{\infty}$ coincides with the 
Verblunsky coefficients of the normalized spectral measure $\tilde{\mu}$. 
In what follows, we also call $\{\alpha(n)\}_{n=1}^{\infty}$ the Verblunsky coefficients of 
$\{X_n\}$.

Our main result, i.e., Theorem \ref{thm:2.1} below, is an explicit representation of 
$\{\alpha(n)\}_{n=1}^{\infty}$. To state it, we need some notation. 
A stationary process $\{X_n\}$ is said to be 
\textit{purely nondeterministic}\/ (PND) if 
$\cap_{n=-\infty}^{\infty}H_{(-\infty,n]}=\{0\}
$, 
or, equivalently, there exists a positive even and 
integrable function $\Delta$ on $(-\pi,\pi]$ such that 
$\int_{-\pi}^{\pi}\vert\log \Delta(\theta)\vert d\theta<\infty$ and 
$
\gamma_n=\int_{-\pi}^{\pi}e^{in\theta}\Delta(\theta)d\theta
$ 
for $n\in\mathbb{Z}$; see Brockwell and Davis (1991, Section 5.7), 
Rozanov (1967, Chapter II) and 
Grenander and Szeg\"o (1958, Chapter 10). 
We call $\Delta$ the \textit{spectral density}\/ of $\{X_n\}$. 
Using $\Delta$, we define the 
\textit{Szeg\"o function}\/ $h$ by
\begin{equation}
h(z):=\sqrt{2\pi}\mbox{exp}\left\{\frac{1}{4\pi}\int_{-\pi}^{\pi}
\frac{e^{i\theta}+z}{e^{i\theta}-z}
\log \Delta(\theta)d\theta\right\}, \qquad z\in\mathbb C,\ \vert z\vert<1.
\label{eq:2.2}
\end{equation}
The function $h(z)$ is 
an outer function in the Hardy space $H^{2}$ 
of class 2 over the unit disk $\vert z\vert<1$. 
Using $h$, we define the MA coefficients $c_n$ and the AR coefficients $a_n$, 
respectively, 
by
\begin{equation}
h(z)=\sum\nolimits_{n=0}^{\infty}c_nz^n, \qquad
-\frac{1}{h(z)}=\sum\nolimits_{n=0}^{\infty}a_nz^n, \qquad \vert z\vert<1;
\label{eq:2.3}
\end{equation}
see Inoue (2000, Section 4) and Inoue and Kasahara (2006, Section 2.2) for background. 
Both $\{c_n\}$ and $\{a_n\}$ are 
real sequences, and $\{c_n\}$ is in $l^2$.

We write $\mathcal{R}_0$ for the class of \textit{slowly varying functions}\/ at infinity:
the class of positive, measurable $\ell$, defined on some 
neighborhood $[A,\infty)$ of infinity, such that
$\lim_{x\to\infty}\ell(\lambda x)/\ell(x)=1$ for all $\lambda>0$; 
see Bingham et al.\ (1989, Chapter 1) for background. 
Among several possible choices of assumption on $\{X_n\}$, as in Inoue and Kasahara (2006), 
we consider
\begin{itemize}
\item[({\bf A1})]$\{X_n\}$ is PND, and 
both $\sum\nolimits_{n=0}^\infty\vert a_n\vert<\infty$ and 
$\sum\nolimits_{n=0}^\infty\vert c_n\vert<\infty$ hold
\end{itemize}
as a standard one for processes with short memory, and 
\begin{itemize}
\item[({\bf A2})]$\{X_n\}$ is PND, and, for some 
$d\in (0,1/2)$ and 
$\ell\in\mathcal{R}_0$, $\{c_n\}$ and $\{a_n\}$ satisfy, respectively,
\[
c_n\sim n^{-(1-d)}\ell(n),
\qquad 
a_n\sim n^{-(1+d)}\frac{1}{\ell(n)}\cdot\frac{d\sin(\pi d)}{\pi},
\qquad n\to\infty
\]
\end{itemize}
as a standard one for processes with long memory. 
Here $p_n\sim q_n$ as $n\to\infty$ means $\lim_{n\to\infty}p_n/q_n=1$.

Recall $\beta_n$ from (\ref{eq:1.1}). 
Notice that the sum in (\ref{eq:1.1}) converges absolutely 
under either (A1) or (A2). 
For $n\in\mathbb{N}\cup\{0\}$, we define 
$\alpha_1(n):=\beta_n$ 
and, for $k=3,5,7,\dots$,
\[
\alpha_k(n)
:=\sum\nolimits_{v_1=0}^\infty \cdots\sum\nolimits_{v_{k-1}=0}^\infty
\beta_{n+v_1}\beta_{n+1+v_1+v_2}
\cdots\beta_{n+1+v_{k-2}+v_{k-1}}\beta_{n+1+v_{k-1}}.
\]
As in the case of $d_k(n,j)$ in Inoue and Kasahara (2006, Section 2.3), the sums 
converge absolutely. 
We write $\sum^{\infty-}$ to indicate that the sum does 
not necessarily converge absolutely, i.e., 
$\sum_{k=m}^{\infty-}:=\lim_{M\to\infty}\sum_{k=m}^{M}$. 

Here is the main result of this paper.

\begin{thm}\label{thm:2.1}
We assume either ${\rm (A1)}$ or ${\rm (A2)}$. 
Then 
$\alpha(n)
=\sum\nolimits_{k=1}^{\infty-}\alpha_{2k-1}(n)$ for $n=2,3,\dots$.
\end{thm}

The proof of Theorem \ref{thm:2.1} is given in Section \ref{sec:3}.

\begin{remark}\label{rem:1}
We assume $\{a_n\}\in\ell^2$. 
As usual, we identify $h$ with its boundary-value function 
$h(e^{i\theta})=\lim_{r\uparrow 1}h(re^{i\theta})$. 
Then, since $h(e^{i\theta})=\sum_{k=0}^{\infty}c_ke^{ik\theta}$ and 
$1/h(e^{i\theta})=-\sum_{k=0}^{\infty}e^{ik\theta}a_ke^{ik\theta}$, 
Parseval's identity yields
\[
\int_{-\pi}^{\pi}e^{-in\theta}\{\overline{h(e^{i\theta})}/h(e^{i\theta})\}
\frac{d\theta}{2\pi}
=-\sum\nolimits_{k=0}^{\infty}c_ka_{k+n}=-\beta_n,\qquad n=0,1,\dots.
\]
Here notice that, in our set-up, $\{c_n\}$ is real. 
Thus $\beta_n$ (or, more precisely, $-\beta_n$) is the $n$-th Fourier 
coefficient of $\bar h/h$. The function $\bar h/h$ is called the \textit{phase 
function}\/ of the process. See Peller (2003, p.\ 405); 
see also Dym and McKean (1976) for its continuous-time analogue.
\end{remark}


\section{Proof of Theorem \ref{thm:2.1}}\label{sec:3}

In this section, we assume either (A1) or (A2). 
For $n\in\mathbb{N}$, 
we can express $P_{[-n,-1]}X_{0}$ uniquely 
in the form
\[
P_{[-n,-1]}X_{0}=\sum\nolimits_{j=1}^n\phi_{n,j}X_{-j}.
\]
We call $\phi_{n,j}$ the \textit{finite predictor coefficients}\/. 
The proof of Theorem \ref{thm:2.1} is based on the explicit representation of $\phi_{n,j}$, 
i.e., (\ref{eq:3.5}) below, and 
the following Szeg\"o recursion (or the Levinson--Durbin algorithm):
\begin{equation}
\phi_{n,j}-\phi_{n+1,j}=\phi_{n,n+1-j}\alpha(n+1),\qquad j=1,\dots,n.
\label{eq:3.1}
\end{equation}
See, e.g., (5.2.4) in Brockwell and Davis (1991) for the latter.

As in Inoue and Kasahara (2006, Section 2.3), we define, for $n,j\in\mathbb{N}\cup\{0\}$,
\[
d_0(n,j)=\delta_{j0},\qquad
d_1(n,j)=\beta_{n+j},\qquad
d_2(n,j)=\sum\nolimits_{v_1=0}^{\infty}\beta_{n+j+v_1}\beta_{n+v_1},\quad 
\]
and
\[
d_k(n,j)
=\sum\nolimits_{v_{1}=0}^{\infty}\cdots \sum\nolimits_{v_{k-1}=0}^{\infty}
\beta_{n+j+v_{k-1}}\beta_{n+v_{k-1}+v_{k-2}}
\cdots\beta_{n+v_{2}+v_{1}}\beta_{n+v_{1}},\quad k\ge 3,
\]
the sums converging absolutely. 
These satisfy the following 
recursion: 
for $n,j\in\mathbb{N}\cup\{0\}$,
\begin{equation}
d_0(n,j)=\delta_{j0},\qquad 
d_{k+1}(n,j)=\sum\nolimits_{v=0}^{\infty}\beta_{n+j+v}d_k(n,v),\qquad k\ge 0.
\label{eq:3.2}
\end{equation}
From the definition of $\alpha_k(n)$ above, we also have
\begin{equation}
\alpha_{2k+1}(n)=\sum\nolimits_{v=0}^{\infty}\beta_{n+v}d_{2k}(n+1,v), \qquad 
n,k\in\mathbb{N}\cup\{0\}.
\label{eq:3.3}
\end{equation}

The next proposition is the key to the proof of Theorem \ref{thm:2.1}.

\begin{prop}\label{prop:3.1}
For $n,j\in\mathbb{N}\cup\{0\}$ and $k\in\mathbb{N}$, we have
\begin{equation}
d_{2k}(n,j)-d_{2k}(n+1,j)
=\sum\nolimits_{l=1}^{k}\alpha_{2k-2l+1}(n)d_{2l-1}(n,j).
\label{eq:3.4}
\end{equation}
\end{prop}

\begin{proof}
Let $n,j\in\mathbb{N}\cup\{0\}$. 
We use mathematical induction on $k$. First, 
since $\alpha_1(n)=\beta_n$ and $d_1(n,j)=\beta_{n+j}$, we have 
$d_2(n,j)
=\sum_{v=n}^{\infty}\beta_v\beta_{v+j}
=\sum_{v=n}^{\infty}\alpha_1(v)d_1(v,j)$, which 
implies (\ref{eq:3.4}) with $k=1$. Next, 
we assume that (\ref{eq:3.4}) holds for $k\in\mathbb{N}$. 
Then, by (\ref{eq:3.2}) and the Fubini--Tonelli theorem, we have 
$
d_{2k+2}(n,j)=\sum_{v_2=0}^{\infty}
[\sum\nolimits_{v_1=n}^{\infty}\beta_{v_2+v_1}\beta_{j+v_1}]
d_{2k}(n,v_2)
$, 
whence $d_{2k+2}(n,j)-d_{2k+2}(n+1,j)=\mathrm{I}+\mathrm{II}$, 
where
\[
\mathrm{I}
:=\sum\nolimits_{v_2=0}^{\infty}
\left[\sum\nolimits_{v_1=0}^{\infty}\beta_{n+v_2+v_1}\beta_{n+j+v_1}\right]
\left[d_{2k}(n,v_2)-d_{2k}(n+1,v_2)\right],\qquad 
\mathrm{II}
:=\sum\nolimits_{v_2=0}^{\infty}\beta_{n+v_2}\beta_{n+j}d_{2k}(n+1,v_2).
\]
By (\ref{eq:3.2}), (\ref{eq:3.4}), and the Fubini--Tonelli theorem,
\begin{align*}
\mathrm{I}
&=\sum\nolimits_{v_2=0}^{\infty}
\left[\sum\nolimits_{v_1=0}^{\infty}\beta_{n+v_2+v_1}\beta_{n+j+v_1}\right]
\sum\nolimits_{l=1}^{k}\alpha_{2k-2l+1}(n)d_{2l-1}(n,v_2)\\
&=\sum\nolimits_{l=1}^{k}\alpha_{2k-2l+1}(n)
\sum\nolimits_{v_1=0}^{\infty}\beta_{n+j+v_1}
\sum\nolimits_{v_2=0}^{\infty}\beta_{n+v_1+v_2}d_{2l-1}(n,v_2)\\
&=\sum\nolimits_{l=1}^{k}\alpha_{2k-2l+1}(n)d_{2l+1}(n,j)
=\sum\nolimits_{l=2}^{k+1}\alpha_{2(k+1)-2l+1}(n)d_{2l-1}(n,j),
\end{align*}
while, by (\ref{eq:3.3}), we have 
$
\mathrm{II}
=[\sum\nolimits_{v_2=0}^{\infty}
\beta_{n+v_2}d_{2k}(n+1,v_2)]\beta_{n+j}
=\alpha_{2k+1}(n)d_1(n,j)
$. 
Thus we obtain (\ref{eq:3.4}) with $k$ replaced by $k+1$, as desired. \qed
\end{proof}

For $n\in\mathbb{N}$ and $j=1,\dots,n$, 
Theorem 2.9 in Inoue and Kasahara (2006) asserts the 
representation
\begin{equation}
\phi_{n,j}
=\sum\nolimits_{k=1}^{\infty-}
\left\{b_{2k-1}(n,j)+b_{2k}(n,n+1-j)\right\},
\label{eq:3.5}
\end{equation}
where
\begin{equation}
b_k(n,j)=c_0\sum\nolimits_{u=0}^\infty a_{j+u}d_{k-1}(n+1,u),
\qquad
k=1,2,\dots.
\label{eq:3.6}
\end{equation}
Using Proposition 3.1, we derive two kinds of difference
equations for $b_k(n,j)$.

\begin{prop}\label{prop:3.2}
For $n,k\in\mathbb{N}$ and $j=1,\dots,n$, we have
\begin{eqnarray}
&&b_{2k+1}(n,j)-b_{2k+1}(n+1,j)
=\sum\nolimits_{l=1}^{k}\alpha_{2k-2l+1}(n+1)b_{2l}(n,j),
\label{eq:3.7}\\
&&b_{2k}(n,n+1-j)-b_{2k}(n+1,n+2-j)
=\sum\nolimits_{l=1}^{k}\alpha_{2k-2l+1}(n+1)b_{2l-1}(n,n+1-j).
\label{eq:3.8}
\end{eqnarray}
\end{prop}

\begin{proof}
From (\ref{eq:3.4}) and (\ref{eq:3.6}), 
we easily obtain (\ref{eq:3.7}) in the following way:
\begin{align*}
b_{2k+1}(n,j)-b_{2k+1}(n+1,j)
&=c_0\sum\nolimits_{u=0}^{\infty}a_{j+u}\{d_{2k}(n+1,u)-d_{2k}(n+2,u)\}\\
&=c_0\sum\nolimits_{u=0}^{\infty}a_{j+u}
\sum\nolimits_{l=1}^{k}\alpha_{2k-2l+1}(n+1)
d_{2l-1}(n+1,u)\\
&=\sum\nolimits_{l=1}^{k}\alpha_{2k-2l+1}(n+1)
\left(c_0\sum\nolimits_{u=0}^{\infty}a_{j+u}
d_{2l-1}(n+1,u)\right)
=\sum\nolimits_{l=1}^{k}\alpha_{2k-2l+1}(n+1)b_{2l}(n,j).
\end{align*}

We turn to (\ref{eq:3.8}).
Since $d_1(n+1,u)=\alpha_1(n+1+u)$ and $b_1(n,j)=c_0a_j$, 
it follows from (\ref{eq:3.6}) that
\[
b_2(n,n+1-j)=c_0\sum\nolimits_{u=0}^{\infty}a_{n+1-j+u}d_1(n+1,u)
=\sum\nolimits_{u=n+1}^{\infty}\alpha_1(u)b_1(n,u-j).
\]
Similarly, 
$
b_2(n+1,n+2-j)=c_0\sum_{u=0}^{\infty}a_{n+2-j+u}d_1(n+2,u)
=\sum_{u=n+2}^{\infty}\alpha_1(u)b_1(n,u-j)$. 
Thus (\ref{eq:3.8}) holds for $k=1$. 
If $k\ge 2$, then, by (\ref{eq:3.2}) and  (\ref{eq:3.6}), we have
\begin{align*}
&b_{2k}(n,n+1-j)
=c_0\sum\nolimits_{u=0}^{\infty}a_{n+1-j+u}
\sum\nolimits_{v=0}^{\infty}\beta_{n+1+u+v}d_{2(k-1)}(n+1,v),\\
&b_{2k}(n+1,n+2-j)
=c_0\sum\nolimits_{u=1}^{\infty}a_{n+1-j+u}
\sum\nolimits_{v=0}^{\infty}\beta_{n+1+u+v}d_{2(k-1)}(n+2,v),
\end{align*}
whence $b_{2k}(n,n+1-j)-b_{2k}(n+1,n+2-j)=\mathrm{I}+\mathrm{II}$ 
with
\begin{align*}
&\mathrm{I}
:=c_0\sum\nolimits_{u=0}^{\infty}
a_{n+1-j+u}\sum\nolimits_{v=0}^{\infty}\beta_{n+1+u+v}
\left\{d_{2(k-1)}(n+1,v)-d_{2(k-1)}(n+2,v)\right\},\\
&\mathrm{II}
:=c_0a_{n+1-j}\sum\nolimits_{v=0}^{\infty}\beta_{n+1+v}d_{2(k-1)}(n+2,v).
\end{align*}
By (\ref{eq:3.2}), (\ref{eq:3.4}) and (\ref{eq:3.6}),
\begin{align*}
\mathrm{I}&=c_0\sum\nolimits_{u=0}^{\infty}a_{n+1-j+u}
\sum\nolimits_{v=0}^{\infty}\beta_{n+1+u+v}
\sum\nolimits_{l=1}^{k-1}\alpha_{2(k-1)-2l+1}(n+1)d_{2l-1}(n+1,v)\\
&=\sum\nolimits_{l=1}^{k-1}\alpha_{2(k-1)-2l+1}(n+1)
\left(c_0
\sum\nolimits_{u=0}^{\infty}a_{n+1-j+u}
\sum\nolimits_{u=0}^{\infty}\beta_{n+1+u+v}d_{2l-1}(n+1,v)\right)\\
&=\sum\nolimits_{l=1}^{k-1}\alpha_{2(k-1)-2l+1}(n+1)b_{2l+1}(n,n+1-j)
=\sum\nolimits_{l=2}^{k}\alpha_{2k-2l+1}(n+1)b_{2l-1}(n,n+1-j),
\end{align*}
while, by (\ref{eq:3.3}) and 
$b_1(n,n+1-j)=ca_{n+1-j}$, we have 
$\mathrm{II}=\alpha_{2k-1}(n+1)b_1(n,n+1-j)$. 
Thus (\ref{eq:3.8}) follows. \qed
\end{proof}

We are now ready to prove Theorem \ref{thm:2.1}.

\begin{proof}[of Theorem \ref{thm:2.1}]
For $n\in\mathbb{N}$ and $j=1,\dots,n$, we have 
$b_1(n,j)-b_1(n+1,j)=c_0a_{j}-c_0a_{j}=0$.
This, together with 
(\ref{eq:3.5}), (\ref{eq:3.7}) and (\ref{eq:3.8}),
yields 
\begin{align*}
\phi_{n,j}-\phi_{n+1,j}
&=\sum\nolimits_{k=1}^{\infty-}
\left\{b_{2k+1}(n,j)-b_{2k+1}(n+1,j)+b_{2k}(n,n+1-j)-b_{2k}(n+1,n+2-j)\right\}\\
&=\sum\nolimits_{k=1}^{\infty-}
\sum\nolimits_{l=1}^{k}\alpha_{2k-2l+1}(n+1)
\left\{b_{2l}(n,j)+b_{2l-1}(n,n+1-j)\right\}\\
&=\sum\nolimits_{l=1}^{\infty-}
\left\{b_{2l}(n,j)+b_{2l-1}(n,n+1-j)\right\}
\sum\nolimits_{k=l}^{\infty-}\alpha_{2k-2l+1}(n+1)
=\left\{\sum\nolimits_{k=1}^{\infty-}
\alpha_{2k-1}(n+1)\right\}\phi_{n,n+1-j}.
\end{align*}
Since $P_{[-n,-1]}X_0\neq 0$, we have $(\phi_{n,1},\dots,\phi_{n,n})\neq 0$. 
Combining these and (\ref{eq:3.1}), we obtain the theorem. \qed
\end{proof}


\section{Applications}\label{sec:4}


\subsection{Short memory processes}\label{subsec:4.1}
In this subsection, we apply Theorem \ref{thm:2.1} to 
the Verblunsky coefficients of short-memory processes.

We define
\[
F(j):=\left\{\sum\nolimits_{v=0}^{\infty}\vert c_v\vert\right\}
\left\{\sum\nolimits_{u=j}^{\infty}\vert a_u\vert\right\}, \qquad 
j=0,1,\dots.
\]
Then $F(j)$ decreases to zero as $j\to\infty$ under (A1). 
Recall $d_k(n,j)$ from Section \ref{sec:3}.

\begin{lem}\label{lem:4.1}
We assume\/ $(\mathrm{A1})$. 
Then $\sum_{u=0}^{\infty}
\vert d_{k}(n,u)\vert\le F(n)^k$ for $k,n\in\mathbb{N}$.
\end{lem}

\begin{proof}Let $n\in\mathbb{N}$. We use induction on $k$. 
Since $d_1(n,u)=\beta_{n+u}$, we have
\[
\sum\nolimits_{u=0}^{\infty}\vert d_1(n,u)\vert
=\sum\nolimits_{u=0}^{\infty}\vert \beta_{n+u}\vert
\le \sum\nolimits_{v=0}^{\infty}\vert c_v\vert
\sum\nolimits_{u=0}^{\infty}\vert a_{n+u+v}\vert
\le F(n).
\]
We assume 
$\sum_{u=0}^{\infty}\vert d_{k}(n,u)\vert\le F(n)^k$ for $k\in\mathbb{N}$. 
Then, by (\ref{eq:3.2}),
\begin{align*}
\sum\nolimits_{u=0}^{\infty}
\vert d_{k+1}(n,u)\vert
&\le \sum\nolimits_{v=0}^{\infty}\vert d_{k}(n,v)\vert\sum\nolimits_{u=0}^{\infty}
\vert\beta_{n+v+u}\vert
\le F(n)\sum\nolimits_{v=0}^{\infty}
\vert d_k(n,v)\vert\le F(n)^{k+1}.
\end{align*}
Thus the inequality also holds for $k+1$. \qed
\end{proof}

Notice that $\{a_n\}\in\ell^1$ implies $a_n\to 0$ as $n\to\infty$. 

\begin{thm}\label{thm:4.2}
We assume {\rm (A1)}. Then, for $N\in\mathbb{N}$ such that 
$F(N+1)<1$, 
the Verblunsky coefficients $\{\alpha(n)\}$ 
satisfy
\[
\vert\alpha(n)\vert\le \frac{\sum_{v=0}^{\infty}\vert c_v\vert}
{1-F(n+1)^2}\left(\max_{j\ge n}\vert a_j\vert\right), \qquad n\ge N.
\]
\end{thm}

\begin{proof}
Recall $\alpha_{2k+1}(n)$ from \S 2. 
We have 
$
\vert \alpha_1(n)\vert
=\vert \beta_n\vert\le  \left(\max_{j\ge n}\vert a_j\vert\right)
\sum\nolimits_{v=0}^{\infty}\vert c_v\vert
$. 
By (\ref{eq:3.3}) and Lemma \ref{lem:4.1}, we also have 
$
\vert \alpha_{2k+1}(n)\vert
\le \sum\nolimits_{v=0}^{\infty}\vert\beta_{n+v}d_{2k}(n+1,v)\vert
\le F(n+1)^{2k}\left(\max_{j\ge n}\vert a_j\vert\right)
\sum\nolimits_{v=0}^{\infty}\vert c_v\vert
$ 
for $n,k\in\mathbb{N}$. 
Choose $N\in\mathbb{N}$ so that $F(N+1)<1$. Then, combining 
the estimates above with Theorem \ref{thm:2.1}, we see that, for $n\ge N$,
\[
\vert\alpha(n)\vert\le \left(\max_{j\ge n}\vert a_j\vert\right)
\sum\nolimits_{v=0}^{\infty}\vert c_v\vert\sum\nolimits_{k=0}^{\infty}F(n+1)^{2k}
=\frac{\sum\nolimits_{v=0}^{\infty}\vert c_v\vert}
{1-F(n+1)^2}\left(\max_{j\ge n}\vert a_j\vert\right).
\]
Thus the theorem follows. \qed
\end{proof}

For example, if, in addition to (A1), $a_n=O(n^{-p})$ as $n\to\infty$ for 
some $p>1$, then $\max_{j\ge n}\vert a_j\vert=O(n^{-p})$, whence, by 
Theorem \ref{thm:4.2}, $\alpha(n)=O(n^{-p})$ as $n\to\infty$.


\subsection{The FARIMA model}\label{subsec:4.2}

For 
$d\in (-1/2, 1/2)$ and $p, q\in\mathbb{N}\cup\{0\}$, 
a stationary process $\{X_n\}$ is said to be a FARIMA$(p,d,q)$ (or fractional 
ARIMA$(p,d,q)$) process if it has a spectral 
density $\Delta$ of 
the form
\[
\Delta(\theta)
=\frac{1}{2\pi}
\frac{\vert \Theta(e^{i\theta})\vert^2}{\vert \Phi(e^{i\theta})\vert^2}
\vert 1-e^{i\theta}\vert^{-2d},\qquad -\pi<\theta\le \pi,
\]
where $\Phi(z)$ and $\Theta(z)$ are polynomials with real coefficients 
of degrees $p$, $q$, respectively, satisfying the following condition: 
$\Phi(z)$ and $\Theta(z)$ have no common zeros, and have no zeros 
in the closed unit disk $\{z\in\mathbb{C}\ \vert\ \vert z\vert\le 1\}$.

In what follows, we assume that $\{X_n\}$ is a FARIMA$(p,d,q)$ process with $0<d<1/2$. 
Then $\{X_n\}$ satisfies (A2) for some constant 
function $\ell$ (cf.\ Inoue (2002, Corollary 3.1)).
Let $\{\alpha(n)\}$ be the Verblunsky coefficients of $\{X_n\}$. 
The aim of this subsection is to apply Theorem \ref{thm:2.1} to $\{\alpha(n)\}$ 
to prove the next theorem.

\begin{thm}\label{thm:4.3}
We have 
$n\alpha(n)=d+O(n^{-d})$ as $n\to\infty$.
\end{thm}

Theorem \ref{thm:4.3} is more precise than Inoue (2008, Theorem 2.5) with $0<d<1/2$, 
in that the former gives an estimate on the remainder term. 
The rest of this subsection is devoted to the proof of 
Theorem \ref{thm:4.3}.

As before, we denote by $\{c_n\}$ and $\{a_n\}$ the 
MA and AR coefficients, respectively, of $\{X_n\}$. 
We also consider a FARIMA$(0,d,0)$ process $\{X_n'\}$ 
satisfying 
$E[(X_n')^2]=\Gamma(1-2d)/\Gamma(1-d)^2$. 
The AR coefficients $\{a_n'\}$ and MA coefficients $\{c_n'\}$ of $\{X'_n\}$ 
are given by
\[
a_n'=\frac{\Gamma(n-d)d}{\Gamma(n+1)\Gamma(1-d)},
\qquad
c_n'=\frac{\Gamma(n+d)}{\Gamma(n+1)\Gamma(d)},\qquad n=0,1,\dots
\]
(see, e.g., Brockwell and Davis (1991, Section 13.2)). 
Notice that $c_n'>0$ for $n\ge 0$ and $a_n'>0$ for $n\ge 1$.
Put
\[
\beta'_n:=\sum\nolimits_{v=0}^{\infty}c_v'a_{n+v}',\qquad n=0,1,\dots.
\]

\begin{lem}\label{lem:4.4}We have 
$\beta'_{n}=\sin(\pi d)/\{\pi(n-d)\}$ for $n=0,1,\dots$.
\end{lem}

\begin{proof}Using the hypergeometric function, we have, for 
$n\ge 0$,
\begin{align*}
\beta'_n
&=\frac{\Gamma(n-d)d}{\Gamma(n+1)\Gamma(1-d)}
\sum_{v=0}^{\infty}\frac{\Gamma(d+v)}{\Gamma(d)}\cdot
\frac{\Gamma(n-d+v)}{\Gamma(n-d)}\cdot\frac{\Gamma(n+1)}{\Gamma(n+1+v)}
\cdot\frac{1}{v!}
=\frac{\Gamma(n-d)d}{\Gamma(n+1)\Gamma(1-d)}F(d,n-d;n+1;1)\\
&=\frac{\Gamma(n-d)d}{\Gamma(n+1)\Gamma(1-d)}
\cdot \frac{\Gamma(n+1)}{\Gamma(n+1-d)\Gamma(1+d)}
=\frac{\sin(\pi d)}{\pi}\cdot\frac{1}{n-d},
\end{align*}
as desired. \qed
\end{proof}

\begin{prop}\label{prop:4.5}
There exist a real sequence $\{\delta_n\}_{n=1}^{\infty}$ and 
$M\in (0,\infty)$ such that 
$\beta_n=\beta'_n\{1+\delta_n\}$ and 
$\vert \delta_n\vert\le Mn^{-d}$ for $n\in\mathbb{N}$.
\end{prop}

\begin{proof}
By Inoue (2002, Lemma 2.2), we have, as $n\to\infty$,
\begin{align*}
\frac{c_n}{n^{d-1}}=\frac{K_1}{\Gamma(d)}+O\left(n^{-1}\right),
\quad
\frac{a_n}{n^{-d-1}}=-\frac{1}{K_1\Gamma(-d)}+O\left(n^{-1}\right),\quad 
\frac{c_n'}{n^{d-1}}=\frac{1}{\Gamma(d)}+O\left(n^{-1}\right),
\quad
\frac{a_n'}{n^{d-1}}=-\frac{1}{\Gamma(-d)}+O\left(n^{-1}\right),
\end{align*}
where $K_1:=\theta(1)/\phi(1)>0$. 
Hence we may write 
$c_n=\{K_1+s_n\}c_n'$ for $n\ge 0$ and 
$a_n=\{(1/K_1)+t_n\}a_n'$ for $n\ge 1$, 
where $\{s_n\}$ and $\{t_n\}$ are sequences satisfying 
$
\vert s_n\vert \le L/(n+1)
$ for $n\ge 0$ and 
$\vert t_n\vert \le L/n$ for $n\ge 1$, for some $L\in (0,\infty)$. 

We have, for $n=1,2,\dots$,
\[
\vert \beta_n-\beta'_n\vert
\le \sum\nolimits_{v=0}^{\infty}\vert s_v\vert c_v'a_{n+v}'
+K_1\sum\nolimits_{v=0}^{\infty}\vert t_{n+v}\vert c_v'a_{n+v}'
+(1/K_1)\sum\nolimits_{v=0}^{\infty}\vert s_vt_{n+v}\vert c_v'a_{n+v}'.
\]
From $c_n'/(n+1)\sim 1/\{n^{2-d}\Gamma(d)\}$ as $n\to\infty$, 
we see that
\[
\sum\nolimits_{v=0}^{\infty}\vert s_v\vert c_v'a_{n+v}'
\le 
L\sum\nolimits_{v=0}^{\infty}\frac{c_v'}{v+1}a_{n+v}'
\sim a_n'L\sum\nolimits_{v=0}^{\infty}\frac{c_v'}{v+1},\qquad n\to\infty.
\]
Hence, using $a_n'\sim\mathrm{constant}\cdot n^{-(1+d)}$ as $n\to\infty$,
we get 
$
\sum\nolimits_{v=0}^{\infty}\vert s_v\vert c_v'a_{n+v}'=O\left(n^{-(1+d)}\right)
$ as $n\to\infty$. 
Similarly, as $n\to\infty$,
\[
\sum\nolimits_{v=0}^{\infty}\vert t_{n+v}\vert c_v'a_{n+v}'=O\left(n^{-(2+d)}\right),
\qquad 
\sum\nolimits_{v=0}^{\infty}\vert s_vt_{n+v}\vert c_v'a_{n+v}'
=O\left(n^{-(2+d)}\right).
\]
Combining these and 
$\beta'_n\sim\pi^{-1}\sin(\pi d)\,n^{-1}$, 
we obtain the proposition. \qed
\end{proof}

\begin{proof}[of Theorem \ref{thm:4.3}]
By Theorem \ref{thm:2.1}, the Verblunsky coefficients 
$\{\alpha(n)\}$ of $\{X_n\}$ and $\{\alpha'(n)\}$ of 
$\{X'_n\}$ admit the representations 
$
\alpha(n)=\sum\nolimits_{k=1}^{\infty-}\alpha_{2k-1}(n)
$ and 
$
\alpha'(n)=\sum\nolimits_{k=1}^{\infty-}\alpha'_{2k-1}(n)
$, respectively, 
where $\alpha_{2k-1}(\cdot)$ are those defined 
for $\{X_n\}$ in \S 2, while 
$\alpha'_{2k-1}(\cdot)$ are their counterparts defined for 
$\{X_n'\}$, that is, 
$\alpha_1'(n)=\beta'_n$ and, 
for $k=3,5,\dots$,
\[
\alpha_k'(n)=
\sum\nolimits_{m_1=0}^{\infty}\cdots\sum_{m_{k-1}=0}^{\infty}\nolimits
\beta'_{n+1+m_1}\beta'_{n+1+m_1+m_2}
\cdots\beta'_{n+1+m_{k-2}+m_{k-1}}\beta'_{n+m_{k-1}}.
\]

For $k\in\mathbb{N}$, we define $\tau_{2k-1}\in (0,\infty)$ by $\tau_{2k-1}:=
(2k-2)!/[\pi2^{2k-2}((k-1)!)^2(2k-1)]$, or by
\begin{equation}
\sum\nolimits_{k=1}^{\infty}\tau_{2k-1}x^{2k-1}=\pi^{-1}\arcsin x,\qquad \vert x\vert<1
\label{eq:4.1}
\end{equation}
(see Inoue and Kasahara (2006, Lemma 3.1) and Inoue (2008, Section 5)). 
Let $M$ be as in Proposition \ref{prop:4.5} and let $r>1$ be chosen so that $r^2\sin(\pi d)<1$. 
Then, as in the proof of Inoue and Kasahara (2006, Proposition 3.2), 
there exists an integer $N$ independent of $k$ such that
\[
1+(M/n^d)\le r,\qquad 
\alpha'_{2k-1}(n)\le \frac{1}{n}\{r\sin(\pi d)\}^{2k-1}\tau_{2k-1},\qquad n\ge N,\ k\ge 1.
\]
By Proposition \ref{prop:4.5}, we have 
$\vert\beta_{n+v}\vert\le \left(1+Mn^{-d}\right)\beta'_{n+v}$ and 
$\vert\beta_{n+v}-\beta'_{n+v}\vert\le Mn^{-d}\beta'_{n+v}$ 
for $n\ge 1$ and $v\ge 0$. 
We also have 
$(1+x)^k-1\le kx(1+x)^k$ for $x\ge 0$. 
Hence, for $n\ge N$, 
\begin{align*}
&\vert \alpha_3(n)-\alpha'_3(n)\vert
\le \sum\nolimits_{m_1=0}^{\infty}\sum\nolimits_{m_2=0}^{\infty}
\vert\beta_{n+1+m_1}-\beta'_{n+1+m_1}\vert\cdot
\vert\beta_{n+1+m_{1}+m_{2}}\vert\cdot\vert\beta_{n+m_{2}}\vert \\
&+\sum\nolimits_{m_1=0}^{\infty}\sum\nolimits_{m_2=0}^{\infty}
\beta'_{n+1+m_1}\vert\beta_{n+1+m_{1}+m_{2}}-\beta'_{n+1+m_{1}+m_{2}}\vert
\cdot\vert\beta_{n+m_{2}}\vert
+\sum\nolimits_{m_1=0}^{\infty}\sum\nolimits_{m_2=0}^{\infty}
\beta'_{n+1+m_1}\beta'_{n+1+m_{1}+m_{2}}
\vert\beta_{n+m_{2}}-\beta'_{n+m_{2}}\vert\\
&\le 
Mn^{-d}\left\{(1+Mn^{-d})^2+(1+Mn^{-d})+1\right\}\alpha'_3(n)
=\left\{(1+Mn^{-d})^3-1\right\}\alpha'_3(n)\\
&\le 3Mn^{-d}(1+Mn^{-d})^3\alpha'_3(n)
\le 3Mn^{-(d+1)}\{r^2\sin(\pi d)\}^3\tau_3.
\end{align*}
In the same way, 
$
\vert \alpha_{2k-1}(n)-\alpha'_{2k-1}(n)\vert
\le (2k-1)Mn^{-(d+1)}\{r^2\sin(\pi d)\}^{2k-1}\tau_{2k-1}
$ 
for $k=1,2,\dots$ and $n\ge N$.

Since $\alpha'(n)=d/(n-d)$ (see Hosking (1981, Theorem 1)), 
it follows that
\[
\left\vert \alpha(n)-\frac{d}{n-d}\right\vert
\le \sum\nolimits_{k=1}^{\infty}\vert \alpha_{2k-1}(n)-\alpha'_{2k-1}(n)\vert
\le n^{-(d+1)}M\sum\nolimits_{k=1}^{\infty}(2k-1)\tau_{2k-1}\{r^2\sin(\pi d)\}^{2k-1}.
\]
By (\ref{eq:4.1}), we have 
$\sum\nolimits_{k=1}^{\infty}(2k-1)\tau_{2k-1}\{r^2\sin(\pi d)\}^{2k-1}<\infty$, 
so that
\[
\alpha(n)=\frac{d}{n-d}+O\left(n^{-(d+1)}\right)
=\frac{d}{n}+O\left(n^{-(d+1)}\right),\qquad n\to\infty.
\]
Thus the theorem follows. \qed
\end{proof}








\end{document}